\documentclass[12pt, reqno]{amsart}

\usepackage{amsmath, amsfonts, amssymb}

\usepackage[latin1]{inputenc}
\usepackage{latexsym}

\usepackage{graphics}
\usepackage{color}

\usepackage{natbib}

\newenvironment{proofsect}[1]
{\vskip0.1cm\noindent{\bf #1.}}

\newtheorem{theorem}{{\sc Theorem}}[section]

\newtheorem{cor}[theorem]{{\sc Corollary}}

\newtheorem{lemma}[theorem]{{\sc Lemma}}

\theoremstyle{remark}
\newtheorem{remark}[theorem]{{\sc Remark}}
\newtheorem{remarks}[theorem]{{\sc Remarks}}

\newcommand{\E}{\mathbb E}
\newcommand{\V}{\mathbb V}
\newcommand{\M}{\mathbb M}

\newcommand{\R}{\mathbb R}

\renewcommand{\P}{\mathbb P}

\renewcommand{\qed}{\hfill\ensuremath{\square}}

\begin{document}
\title[Moderate deviations in random graphs and Bernoulli random matrices]{Moderate deviations in a 
random graph and for the spectrum of Bernoulli random matrices}
\bigskip
\author[Hanna D\"oring, Peter Eichelsbacher]{}

\maketitle
\thispagestyle{empty}
\vspace{0.2cm}

\centerline{\sc Hanna D\"oring\footnote{Ruhr-Universit\"at Bochum, Fakult\"at f\"ur Mathematik, 
NA 3/67, D-44780 Bochum, Germany, {\tt hanna.doering@ruhr-uni-bochum.de  }}, 
Peter Eichelsbacher\footnote{Ruhr-Universit\"at Bochum, Fakult\"at f\"ur Mathematik, 
NA 3/68, D-44780 Bochum, Germany, {\tt peter.eichelsbacher@ruhr-uni-bochum.de  }}}
\vspace{2 cm}

\pagenumbering{roman}
\maketitle

\pagenumbering{arabic}
\pagestyle{headings}
\bigskip

{ {\bf Abstract:} We prove a moderate deviation principle for subgraph count statistics of Erd\H{o}s-R{\'e}nyi random graphs. 
This is equivalent in showing a moderate deviation principle for the trace of a power of a Bernoulli random matrix. 
It is done via an estimation of the log-Laplace transform and the G\"artner-Ellis theorem. We obtain upper bounds on the upper
tail probabilities of the number of occurrences of small subgraphs.
The method of proof is used to 
show supplemental moderate deviation principles for a class of symmetric statistics, including non-degenerate U-statistics 
with independent or Markovian entries.}
\bigskip
\bigskip

%%%%%%%%%%%%%%%%%%%%%%%%%%%%%%%%%%%%%%%%%%%%%%%%%%%%%%%%%%%%%%%%%%%%%%%%%%%%%%%
\section{Introduction}

\subsection{Subgraph-count statistics}
Consider an Erd\H{o}s-R{\'e}nyi random graph with $n$ vertices, where
for all $\left(n \atop 2\right)$ different pairs of vertices the
existence of an edge is decided by an independent Bernoulli experiment
with probability $p$. For each $i\in\{1,\dots,\left({{n}\atop{2}}\right)\}$,
let $X_i$ be the random variable determining if the edge $e_i$ is present, i.e. $P(X_i=1)=1-P(X_i=0)=p(n) =:p$.
The following statistic counts the number of
subgraphs isomorphic to a fixed graph $G$ with $k$ edges and $l$ vertices
$$
W=
\sum_{1\leq \kappa_1<\dots < \kappa_k \leq \left({{n}\atop{2}}\right)}
  1_{\{(e_{\kappa_1},\dots,e_{\kappa_k})\sim G\}} \left(\prod_{i=1}^k X_{\kappa_i} \right)
\:.
$$
Here $(e_{\kappa_1}, \ldots, e_{\kappa_k})$ denotes the graph with edges
$e_{\kappa_1}, \ldots, e_{\kappa_k}$ present and $A \sim G$ denotes the fact that the subgraph $A$
of the complete graph is isomorphic to $G$.
We assume $G$ to be a graph without isolated vertices and to consist
of $l\geq 3$ vertices and $k\geq2$ edges.
Let the constant $a := \rm{aut}(G)$ denote the order of the automorphism group of $G$.
The number of copies of $G$ in $K_n$, the complete graph with $n$
vertices and $\left(n \atop 2\right)$ edges, is given by
$\left(n \atop l\right) l!/a$ and the expectation of $W$ is equal to
$$
\E[W] = \frac{\left(n \atop l\right) l!}{a} p^k = {\mathcal O}(n^l p^k)
\:.
$$
It is easy to see that $P(W> 0)= o(1)$ if $p\ll n^{-l/k}$.
Moreover, for the graph property that $G$ is a subgraph, the probability that
a random graph possesses it jumps from $0$ to $1$ at the threshold probability
$n^{-1/m(G)}$, where
$$
m(G)= \max \left\{ \frac{e_H}{v_H} : H\subseteq G, v_H >0 \right\},
$$
$e_H, v_H$ denote the number of edges and vertices of $H\subseteq G$,
respectively, see \cite{JLR:2000}.

Limiting Poisson and normal distributions for subgraph counts were studied for 
probability functions $p=p(n)$. 
For $G$ be an arbitrary graph, Ruci{\'n}ski proved in \cite{Rucinski:1988} that $W$ is Poisson convergent
if and only if 
\begin{equation*}
n p^{d(G)}\stackrel{n\to\infty}{\longrightarrow}  0
\quad\text{or}\quad
n^{\beta(G)} p \stackrel{n\to\infty}{\longrightarrow}  0\:.
\end{equation*}
Here $d(G)$ denotes the density of the graph $G$ and 
$$
\beta(G) := \max \biggl\{ \frac{v_G-v_H}{e_G-e_H} : H \subset G \, \biggr\}.
$$

Consider
\begin{equation} \label{NWnorm}
c_{n,p} := \left({{n-2}\atop{l-2}}\right) \frac{2k}{a} (l-2)!
  \sqrt{\left({{n}\atop{2}}\right)p(1-p)} p^{k-1}
\end{equation}
and
\begin{equation}
Z:= \frac{W-\E(W)}{c_{n,p}} =
\frac{
\sum_{1\leq \kappa_1<\dots < \kappa_k \leq \left({{n}\atop{2}}\right)}
  1_{\{(e_{\kappa_1},\dots,e_{\kappa_k})\sim G\}} \left(\prod_{i=1}^k X_{\kappa_i} -
    p^k \right)}
{\left({{n-2}\atop{l-2}}\right) \frac{2k}{a} (l-2)!
  \sqrt{\left({{n}\atop{2}}\right)p(1-p)} p^{k-1}}
\:.
\end{equation}
$Z$ has asymptotic standard normal distribution, if
$n p^{k-1}\stackrel{n\to\infty}{\longrightarrow}  \infty$
and $n^2 (1-p) \stackrel{n\to\infty}{\longrightarrow} \infty$,
see Nowicki, Wierman, \cite{NowickiWierman:1988}. For $G$ be an arbitrary graph with at least
one edge, Ruci{\'n}ski proved in \cite{Rucinski:1988} that
$\frac{W-\E(W)}{\sqrt{\V(W)}}$
converges in distribution to a standard normal distribution if and only if
\begin{equation}
n p^{m(G)}\stackrel{n\to\infty}{\longrightarrow}  \infty
\quad\text{and}\quad
n^2 (1-p) \stackrel{n\to\infty}{\longrightarrow}  \infty\:.
\end{equation}
Here and in the following $\V$ denotes the variance of the corresponding random variable.
Ruci{\'n}ski closed the book proving asymptotic normality in applying the method of moments.
One may wonder about the normalization \eqref{NWnorm} used in \cite{NowickiWierman:1988}.
The subgraph count $W$ is a sum of dependent random variables, for which the exact calculation
of the variance is tedious. In \cite{NowickiWierman:1988}, the authors approximated $W$ by a
projection of $W$, which is a sum of independent random variables. For this sum the variance
calculation is elementary, proving the denominator \eqref{NWnorm} in the definition of $Z$.
The asymptotic behaviour of the variance of $W$ for any $p=p(n)$ is summarized in Section 2
in \cite{Rucinski:1988}.
The method of martingale differences used by Catoni in \cite{Catoni:2003} enables on the conditions
$n p^{3(k-\frac{1}{2})}\stackrel{n\to\infty}{\longrightarrow}  \infty$
and $n^2 (1-p) \stackrel{n\to\infty}{\longrightarrow} \infty$ to give an alternative proof of the
central limit theorem, see remark \ref{CLTsubgraph}.

A common feature is to prove large and moderate deviations, namely, the asymptotic computation
of small probabilities on an exponential scale. Considering the moderate scale is the interest 
in the transition from a result of convergence in distribution like a central limit theorem-scaling
to the large deviations scaling. Interesting enough proving that the subgraph count random variable
$W$ satisfies a large or a moderate deviation principle is an unsolved problem up to now.
The main goal of this paper is to prove a moderate deviation principle
for the rescaled $Z$, filling a substantial gap in the literature on asymptotic subgraph count
distributions, see Theorem \ref{theorem1}.
Before we recall the definition of a moderate deviation principle and state our result,
let us remark, that exponentially small probabilities have been studied extensively in the literature.
A famous upper bound for {\it lower tails} was proven by Janson \cite{Janson:1990}, applying the 
FKG-inequality. This inequality leads to good upper bounds for the probability of nonexistence $W=0$.
Upper bounds for {\it upper tails} were derived by Vu \cite{Vu:2001}, Kim and Vu \cite{KimVu:2004} and  recently by Janson,
Oleskiewicz and Ruci{\'n}ski \cite{JansonOleskiewiczRucinski:2004} and in \cite{JansonRucinski:2004} by Janson and Ruci{\'n}ski.
A comparison of seven techniques proving bounds for {\it the infamous} upper tail can be found in \cite{JansonRucinski:2002}.
In Theorem \ref{concentration} we also obtain upper bounds on the upper tail probabilities of $W$. 
\medskip

Let us recall the definition of a large deviation principle (LDP). A
sequence of probability measures $\{(\mu_n), n\in \mathbb N\}$ on a
topological space $\mathcal X$ equipped with a $\sigma$-field $\mathcal
B$ is said to satisfy the LDP with speed $s_n\nearrow \infty$ and
good rate function $I(\cdot)$ if the level sets $\{x: I(x)\leq
\alpha\}$ are compact for all $\alpha\in[0,\infty)$ and for all
$\Gamma\in\mathcal B$ the lower bound
$$
\liminf_{n\to\infty}  \frac{1}{s_n} \log \mu_n(\Gamma)
\geq - \inf_{x\in \operatorname{int}(\Gamma)} I(x)
$$
and the upper bound
$$
\limsup_{n\to\infty}  \frac{1}{s_n} \log \mu_n(\Gamma)
\leq - \inf_{x\in \operatorname{cl}(\Gamma)} I(x)
$$
hold. Here $\operatorname{int}(\Gamma)$ and
$\operatorname{cl}(\Gamma)$ denote the interior and closure of
$\Gamma$ respectively. We say a sequence of random variables satisfies
the LDP when the sequence of measures induced by these variables
satisfies the LDP. Formally a moderate deviation principle is nothing
else but the LDP. However, we will speak about a moderate deviation 
principle (MDP) for a sequence of random variables, whenever the scaling
of the corresponding random variables is between that of an ordinary Law
of Large Numbers and that of a Central Limit Theorem.

In the following, we state one of our main results, a moderate deviation
principle for the rescaled subgraph count statistic $W$ when $p$ is fixed, 
and when the sequence $p(n)$ converges to 0 or 1 sufficiently slowly.

\begin{theorem}\label{theorem1}
Let $G$ be a fixed graph without isolated vertices, consisting of $k \geq 2$ edges and $l\geq 3$ vertices. The
sequence $(\beta_n)_n$ is assumed to be increasing with
\begin{equation}\label{beta}
n^{l-1} p^{k-1} \sqrt{ p(1-p) } \ll \beta_n
\ll n^{l} \left( p^{k-1} \sqrt{p(1-p)} \right)^4\:.
\end{equation}
Then the sequence $(S_n)_n$ of subgraph count statistics
$$
S:=S_n:=
\frac{1}{\beta_n} \sum_{1\leq \kappa_1<\dots < \kappa_k \leq \left({{n}\atop{2}}\right)}
  1_{\{(e_{\kappa_1},\dots,e_{\kappa_k})\sim G\}} \left(\prod_{i=1}^k X_{\kappa_i} -
    p^k \right)
$$
satisfies a moderate deviation principle with speed
\begin{equation}\label{gamma}
s_n = \frac{\left(\frac{2k}{a} (l-2)!\right)^2 \beta_n^2}{c_{n,p}^2}
=\frac{1}{\left({{n-2}\atop{l-2}}\right)^2 \left({{n}\atop{2}}\right)} \frac{1}{p^{2k-1}(1-p)} \beta_n^2
\end{equation}
and rate function $I$ defined by
\begin{equation}\label{ratef}
I(x)= \frac{x^2}{2\bigl(\frac{2k}{a} (l-2)!\bigr)^2}\:.
\end{equation}
\end{theorem}

\begin{remarks}
\begin{enumerate}
\item
Using $\left({{n-2}\atop{l-2}}\right)^2 \left({{n}\atop{2}}\right)
\leq n^{2(l-1)}$, we obtain
$s_n \geq \left( \frac{\beta_n}{n^{l-1} p^{k-1} \sqrt{p (1-p)}} \right)^2$;
therefore the condition
$$
n^{l-1} p^{k-1} \sqrt{p (1-p)} \ll \beta_n
$$
implies that $s_n$ is growing to infinity as $n\to\infty$ and hence
is a speed.

\item
If we choose $\beta_n$ such that
$
\beta_n
\ll n^{l} \left( p^{k-1} \sqrt{p (1-p)} \right)^4
$
and using the fact that $s_n$ is a speed implies that
\begin{equation}\label{condition}
n^2 p^{6k-3} (1-p)^3
\stackrel{n\to\infty}{\longrightarrow}
\infty\:.
\end{equation}
This is a necessary but not a sufficient condition on \eqref{beta}.
\end{enumerate}
\end{remarks}

The approach to prove Theorem \ref{theorem1} yields additionally
to a central limit theorem for $Z=\frac{W-\E W}{c_{n,p}}$,
see remark \ref{CLTsubgraph}, and to a
concentration inequality for $W - \E W$:

\begin{theorem}\label{concentration}
Let $G$ be a fixed graph without isolated vertices, consisting of $k \geq 2$ edges
and $l \geq 3$ vertices and let $W$ be the number of copies of $G$.
Then for every $\varepsilon>0$
\begin{equation*}
P(W-\E W \geq \varepsilon \E W)
\leq
\exp\left( - \frac{const. \varepsilon^2 n^{2l} p^{2k}}{n^{2l-2} p^{2k-1}(1-p) + const. \varepsilon n^{2l-2} p^{1-k} (1-p)^{-1}}
\right)
\:,
\end{equation*}
where $const.$
are only depending on $l$ and $k$.
\end{theorem}
We will give a proof of Theorem \ref{theorem1} and Theorem \ref{concentration} in the end of section
\ref{sectionsubgraphs}.

\begin{remark}
Let us consider the example of counting triangles: $l=k=3$, $a=6$.
The necessary condition \eqref{condition} of the moderate deviation
principle turns to
$$
n^2 p^{15}\longrightarrow \infty
\quad \text{and}\quad n^2(1-p)^3 \longrightarrow \infty
\quad\text{as } n\to\infty\:.
$$
This can be compared to the expectedly weaker necessary and sufficient
condition for the central limit theorem for $Z$ in \cite{Rucinski:1988}:
$$
n p \longrightarrow \infty
\quad \text{and}\quad n^2(1-p) \longrightarrow \infty
\quad\text{as }n\to\infty.
$$
The concentration inequality in Theorem \ref{concentration} for
triangles turns to
\begin{equation*}
P(W-\E W \geq \varepsilon \E W)
\leq
\exp\left( - \frac{const. \varepsilon^2 n^6 p^6}{n^4 p^5 (1-p) + const. \varepsilon n^4 p^{-2} (1-p)^{-1}} \right)
\quad \forall\varepsilon>0\:.
\end{equation*}
Kim and Vu showed in \cite{KimVu:2004} for all $0<\varepsilon\leq 0.1$
and for $p\geq \frac{1}{n} \log{n}$, that
$$
P\left(
\frac{W-\E W}{\varepsilon p^3 n^3} \geq 1
\right)
\leq e^{-\Theta( p^2 n^2 )} \:.
$$
As we will see in the proof of Theorem \ref{concentration}, the bound for $d(n)$ in \eqref{defdn} leads to an additional term of order $n^2 p^8$. 
Hence in general our bounds are not optimal. Optimal bounds were obtained only for some subgraphs. Our
concentration inequality can be compared with the bounds in \cite{JansonRucinski:2002}, which we leave
to the reader. 
\end{remark}

\subsection{Bernoulli random matrices}
Theorem \ref{theorem1} can be reformulated as a moderate deviation principle for traces of a power
of a Bernoulli random matrix. 

\begin{theorem} \label{cor}
Let $X=(X_{i j})_{i,j}$ be a symmetric $n\times n$-matrix of independent
real-valued random variables, Bernoulli-distributed with probability
\begin{equation*}
P(X_{i j}=1) = 1-P(X_{i j}=0)=p(n), \ i < j
\end{equation*}
and $P(X_{ii}=0)=1$, $i=1, \ldots, n$.
Consider for any fixed $k \geq 3$ the trace of the matrix to the power $k$
\begin{equation}
\text{Tr}(X^k)
= \sum_{i_1,\dots, i_{k}=1}^n X_{i_1 i_2} X_{i_2 i_3} \cdots X_{i_{k} i_1}
\:.
\end{equation}
Note that $\text{Tr}(X^k)= 2 \, W$, for $W$ counting circles
of length $k$ in a random graph. We obtain that the sequence $(T_n)_n$ with
\begin{equation}
T_n :=  \frac{\text{Tr}(X^k) - \E[\text{Tr}(X^k)]}{2 \beta_n}
\end{equation}
satisfies a moderate deviation principle for any $\beta_n$ satisfying \eqref{beta} with $l=k$ and with rate function \eqref{ratef}
with $l=k$ and $a=2k$:
\begin{equation}
I(x)= \frac{x^2}{2 \left((k-2)!\right)^2}\:.
\end{equation}
\end{theorem}

\begin{remark} \label{remark1}
The following is a famous open problem in random matrix theory: Consider $X$ to be a symmetric 
$n \times n$ matrix with entries $X_{ij}$ ($i \leq j$) being i.i.d., satisfying some exponential integrability. 
The question is to prove for any fixed $k \geq 3$ a LDP for 
$$
\frac{1}{n^k} \text{Tr}(X^k)
$$
and the MDP for 
\begin{equation} \label{mdpremark}
\frac{1}{\beta_n(k)} \bigl( \text{Tr}(X^k) -\E[\text{Tr}(X^k)] \bigr) 
\end{equation}
for a properly chosen sequence $\beta_n(k)$. For $k=1$ the LDP in question immediately follows from Cram\'er's theorem (see \cite[Theorem 2.2.3]{DemboZeitouni:1998}),
since 
$$
\frac 1n \text{Tr}(X) = \frac 1n \sum_{i=1}^n X_{ii}.
$$ 
For $k=2$, notice that
$$
\frac{1}{n^2} \text{Tr}(X^2) = \frac{2}{n^2} \sum_{i < j} X_{ij}^2 + \frac{1}{n^2} \sum_{i=1}^n X_{ii}^2 =: A_n + B_n.
$$
By Cram\'er's theorem we know that $(\tilde{A}_n)_n$ with $\tilde{A}_n := \frac{1}{{n \choose 2}} \sum_{i <j} X_{ij}^2$ satisfies
the LDP, and by Chebychev's inequality we obtain for any $\varepsilon >0$
$$
\limsup_{n \to \infty} \frac 1n \log P(|B_n| \geq \varepsilon) = -\infty.
$$
Hence $(A_n)_n$ and $(\frac{1}{n^2} \text{Tr}(X^2))_n$ are exponentially equivalent (see \cite[Definition 4.2.10]{DemboZeitouni:1998}).
Moreover $(A_n)_n$ and $(\tilde{A}_n)_n$ are exponentially equivalent, since Chebychev's inequality leads to
$$
\limsup_{n \to \infty} \frac 1n \log P(|A_n - \tilde{A}_n| > \varepsilon) = 
\limsup_{n \to \infty} \frac 1n \log P \biggr( |\sum_{i<j} X_{ij}^2| \geq \varepsilon \frac{n^2(n-1)}{2} \biggl) = -\infty.
$$
Applying Theorem 4.2.13 in \cite{DemboZeitouni:1998}, we obtain the LDP for $(1/n^2 \text{Tr}(X^2))_n$ under exponential integrability.
For $k \geq 3$, proving the LDP for $(1/n^k \text{Tr}(X^k))_n$ is open, even in the Bernoulli case. For Gaussian entries $X_{ij}$ with mean 0 and variance
$1/n$, the LDP for the sequence of empirical measures of the corresponding eigenvalues $\lambda_1, \ldots, \lambda_n$, e.g. 
$$
\frac 1n \sum_{i=1}^n \delta_{\lambda_i},
$$
has been established by Ben Arous and Guionnet in \cite{BenArous/Guionnet:1997}. 
Although one has the representation
$$
\frac{1}{n^k} \text{Tr}(X^k) = \frac{1}{n^{k/2}} \text{Tr} \biggl( \frac{X}{\sqrt{n}} \biggr)^k = \frac{1}{n^{k/2}} \sum_{i=1}^n \lambda_i^k,
$$
the LDP cannot be deduced from the LDP of the empirical measure by the contraction principle \cite[Theorem 4.2.1]{DemboZeitouni:1998},
because $x \to x^k$ is not bounded in this case.
\end{remark}

\begin{remark}
Theorem \ref{cor} told us that in the case of Bernoulli random variables $X_{ij}$, the MDP for \eqref{mdpremark} holds for any 
$k \geq 3$. For $k=1$ and $k=2$, the MDP for \eqref{mdpremark} holds for arbitrary i.i.d. entries
$X_{ij}$ satisfying some exponential integrability: 
For $k=1$ we choose 
$\beta_n(1) := a_n$ with $a_n$ any sequence with $\lim_{n \to \infty} \frac{\sqrt{n}}{a_n}=0$ 
and $\lim_{n \to \infty} \frac{n}{a_n} = \infty$. For
$$
\frac{1}{a_n} \sum_{i=1}^n (X_{ii} - \E (X_{ii}))
$$
the MDP holds with rate $x^2/(2 \V(X_{11}))$ and speed $a_n^2/n$, see Theorem 3.7.1 in \cite{DemboZeitouni:1998}.
In the case of Bernoulli random variables, we choose $\beta_n(1) = a_n$ with $(a_n)_n$ any sequence with
$$
\lim_{n \to \infty} \frac{\sqrt{n p (1-p)}}{a_n} =0 \,\, \text{and} \,\, \lim_{n \to \infty} \frac{ n \sqrt{p(1-p)}}{a_n} = \infty
$$
and $p=p(n)$. Now $(\frac{1}{a_n} \sum_{i=1}^n (X_{ii} - \E(X_{ii})))_n$ satisfies the MDP with rate function $x^2/2$
and speed
$$
\frac{a_n^2}{n p(n) (1- p(n))}.
$$
Hence, in this case $p(n)$ has to fulfill the condition $n^2 p(n) (1-p(n)) \to \infty$.

For $k=2$, we choose $\beta_n(2) = a_n$ with $a_n$ being any sequence with $\lim_{n \to \infty} \frac{n}{a_n} =0$ and 
$\lim_{n \to \infty} \frac{n^2}{a_n} = \infty$. Applying Chebychev's inequality and exponential equivalence arguments
similar as in Remark \ref{remark1}, we obtain the MDP for
$$
\frac{1}{a_n} \sum_{i,j=1}^n (X_{ij}^2 - \E(X_{ij}^2))
$$
with rate $x^2/(2 \V(X_{11}))$ and speed $a_n^2/n^2$.The case of Bernoulli random variables
can be obtained in a similar way.
\end{remark}

\begin{remark}
For $k \geq 3$ we obtain the MDP with $\beta_n = \beta_n(k)$
such that
$$
n^{k-1} \, p(n)^{k-1} \, \sqrt{p(n)(1-p(n))} \ll \beta_n \ll n^k \bigl( p(n)^{k-1} \sqrt{p(n)(1-p(n))} \bigr)^4.
$$ 
Considering a fixed $p$, the range of $\beta_n$ is what we should expect: $n^{k-1} \ll \beta_n \ll n^k$.
But we also obtain the MDP for functions $p(n)$.
In random matrix theory, Wigner 1959 analysed Bernoulli random matrices in Nuclear Physics. 
Interestingly enough, a moderate deviation principle for the empirical mean of the eigenvalues
of a random matrix is known only for symmetric matrices with Gaussian entries and for
non-centered Gaussian entries, respectively, see \cite{Dembo/Guionnet/Zeitouni:2003}. The proofs
depend on the existence of an explicit formula for the joint distribution of the eigenvalues or on 
corresponding matrix-valued stochastic processes.
\end{remark}

\subsection{Symmetric Statistics}
On the way of proving Theorem \ref{theorem1}, we will apply a nice result of Catoni \cite[Theorem 1.1]{Catoni:2003}. Doing
so, we recognized, that Catoni's approach lead us to a general approach proving a moderate deviation principle
for a rich class of statistics, which -without loss of generality- can be assumed to be {\it symmetric statistics}. 
Let us make this more precise. In \cite{Catoni:2003}, non-asymptotic bounds of the $log$-Laplace transform of
a function $f$ of $k(n)$ random variables $X := (X_1, \ldots, X_{k(n)})$ lead to concentration inequalities.
These inequalities can be obtained for independent random variables or for Markov chains. It is assumed
in \cite{Catoni:2003} that the {\it partial finite differences of order one and two} of $f$ are suitably bounded.
The line of proof is a combination of a martingale difference approach and a Gibbs measure philosophy.

\medskip

Let $(\Omega, {\mathcal A})$ be the product of measurable spaces $\otimes_{i=1}^{k(n)} ({\mathcal X}_i, {\mathcal B}_i)$ 
and $\P = \otimes_{i=1}^{k(n)} \mu_i$ be a product probability measure on $(\Omega, {\mathcal A})$. Let
$X_1,\dots,X_{k(n)}$ take its values in $(\Omega, {\mathcal A})$ and assume that $(X_1, \ldots, X_{k(n)})$ 
is the canonical process. Let $(Y_1,\dots,Y_{k(n)})$ be an
independent copy of $X:=(X_1,\dots,X_{k(n)})$ such that $Y_i$ is distributed according to $\mu_i$, $i=1, \ldots, k(n)$.
The function $f: \Omega \to \mathbb R$ is assumed to be {\it bounded and measurable}.

Let $\Delta_i f(x_1^{k(n)};y_i)$ denote the {\it partial
difference of order one} of $f$ defined by
\begin{align*}
\Delta_i f(x_1^{k(n)};y_i)
: = &\ f(x_1,\dots,x_{k(n)})-f(x_1,\dots,x_{i-1},y_i,x_{i+1},\dots,x_{k(n)})\:,
\end{align*}
where
$x_1^{k(n)}:=(x_1,\dots,x_{k(n)}) \in \Omega$ and
$y_i\in{\mathcal X}_i$. Analogously we define for $j<i$
and $y_j\in{\mathcal X}_j$ the {\it partial difference of order two}
\begin{align*}
\Delta_i \Delta_j f(x_1^{k(n)};y_j, y_i)
:=&\Delta_i f(x_1^{k(n)};y_i)-f(x_1,\dots,x_{j-1},y_j,x_{j+1},\dots,x_{k(n)})
\\
&{}+ f(x_1,\dots,x_{j-1},y_j,x_{j+1},\dots,x_{i-1},y_i,x_{i+1},\dots,x_{k(n)})
\:.
\end{align*}

Now we can state our main theorem. If the random variables are independent 
%or from a Markov chain 
and if the partial finite differences of the first and second order of $f$ are 
suitably bounded, then $f$, properly rescaled, satisfies
the MDP:

\begin{theorem}\label{metathm}
In the above setting assume that the random variables in $X$ are independent.
Define $d(n)$ by
\begin{equation} \label{defdn}
d(n):= \sum_{i=1}^{k(n)} |\Delta_i f(X_1^{k(n)};Y_i)|^2 \left(
\frac{1}{3} |\Delta_i f(X_1^{k(n)};Y_i)|
+ \frac{1}{4}\sum_{j=1}^{i-1} |\Delta_i \Delta_j f(X_1^{k(n)};Y_j,Y_i)|
\right).
\end{equation}
Moreover let there exist two sequences $(s_n)_n$ and $(t_n)_n$ such that
\begin{enumerate}
\item \label{vier}
$\displaystyle \frac{s_n^2}{t_n^3} d(n)
\stackrel{n\to\infty}{\longrightarrow} 0$
for all $\omega\in\Omega$ and
\item \label{drei}
$\displaystyle
\frac{s_n}{t_n^2} \V f(X) \stackrel{n\to\infty}{\longrightarrow} C>0$
for the variance of $f$.
\end{enumerate}
Then the sequence of random variables
$$
\left(\frac{f(X)-\E\bigl[f(X)\bigr]}{t_n}\right)_n
$$
satisfies a moderate deviation principle with speed $s_n$ and
rate function $\frac{x^2}{2 C}$.
\end{theorem}

In Section \ref{Catoni} we are going to prove Theorem \ref{metathm} via the G\"artner-Ellis 
theorem. In \cite{Catoni:2003} an inequality has been proved which allows to
relate the logarithm of a Laplace transform with the expectation and
the variance of the observed random variable.
Catoni proves a similar result for the logarithm of a Laplace transform
of random variables with Markovian dependence. One can find a different
$d(n)$ in \cite[Theorem 3.1]{Catoni:2003}. To simplify notations we did
not generalize Theorem \ref{metathm}, but the proof can be adopted
immediately. 
In Section \ref{Examples} we obtain moderate deviations
for several symmetric statistics, including the sample mean and
$U$-statistics with independent and Markovian entries. In Section
\ref{sectionsubgraphs} we proof Theorem \ref{theorem1} and \ref{concentration}.

%%%%%%%%%%%%%%%%%%%%%%%%%%%%%%%%%%%%%%%%%%%%%%%%%%%%%%%%%%%%%%%%%%%%%%%%
\section{Moderate Deviations via Laplace Transforms} \label{Catoni}
Theorem \ref{metathm} is an application of the following theorem:

\begin{theorem}(Catoni, 2003)\label{thmcatoni}\\
In the setting of Theorem \ref{metathm}, assuming that the random variables in $X$ are independent, one obtains
for all $s\in{\mathbb R_+}$,
\begin{align}\label{catineq}
&\Bigl| \log{\E\exp{\bigl(s f(X)\bigr)}}-s\E\bigl[f(X)\bigr]
-\frac{s^2}{2}\V f(X)\Bigr| \leq s^3 d(n)\\
=&
\sum_{i=1}^{k(n)} \frac{s^3}{3} |\Delta_i f(X_1^{k(n)};Y_i)|^3
+ \sum_{i=1}^{k(n)} \sum_{j=1}^{i-1} \frac{s^3}{4} |\Delta_i f(X_1^{k(n)};Y_i)|^2 |\Delta_i \Delta_j f(X_1^{k(n)};Y_j,Y_i)|\:. \nonumber
\end{align}
\end{theorem}

\begin{proofsect}{Proof of Theorem~\ref{thmcatoni}} 
We decompose $f(X)$ into martingale differences
\begin{equation*}
F_i(f(X))
=\E\bigl[f(X)\big|X_1,\dots,X_i\bigr]-\E\bigl[f(X)\big|X_1,\dots,X_{i-1}\bigr]
\quad\text{, for all }i\in\left\{1,\dots,{k(n)}\right\}
\:.
\end{equation*}
The variance can be represented by
$\V f(X)=\displaystyle \sum_{i=1}^{k(n)} \E\Bigl[\bigl(F_i(f(X))\bigr)^2\Bigr]$.

Catoni uses the triangle inequality and compares the two terms 
$\log{\E e^{s f(X)-s \E[f(X)]}}$ and $\frac{s^2}{2} \V f(X)$
to the above representation of the variance with respect to the Gibbs measure with
density
$$
dP_W := \frac{e^W}{\E[e^W]} \, dP,
$$
where $W$ is a bounded measurable function of $(X_1, \ldots, X_{k(n)})$. 
We denote an expectation due to this Gibbs measure by $\E_W$, e.g.
$$
\E_W[X] :=\frac{\E[X \exp{(W)}]}{\E[\exp{(W)}]}\:.
$$

On the one hand Catoni bounds the difference
$$\Bigl| \log{\E e^{s f(X)-s \E[f(X)]}}
	-\frac{s^2}{2} \sum_{i=1}^{k(n)}
	\E_{s \E\bigl[f(X)-\E[f(X)]\big|X_1,\dots,X_{i-1}\bigr]}
	\bigl[\bigl(F_i\bigl(f(X)\bigr)\bigr)^2\bigr]\Bigr|
$$
via partial integration:
\begin{eqnarray*}
&& \Bigl| \log \E e^{s \bigl(f(X)-\E[f(X)]\bigr)}-
\frac{s^2}{2}\sum_{i=1}^{k(n)} \E_{s \E\bigl[f(X)-\E[f(X)]\big|X_1,\dots,X_{i-1}\bigr]}\bigl[F_i^2(f(X))\bigr]\Bigr|\\
&=&
\Bigl|\sum_{i=1}^{k(n)} \int_0^{s} \frac{(s-\alpha)^2}{2} \M^3_{s \E\bigl[f(X)-\E[f(X)]\big|X_1,\dots,X_{i-1}\bigr]+\alpha F_i(f(X))}[F_i(f(X))] d\alpha\Bigr|
\:,
\end{eqnarray*}
where $\M^3_U[X] := \E_U \bigl[\bigl(X-\E_U[X]\bigr)^3\bigr]$ for a bounded measurable function $U$ of $(X_1, \ldots, X_{k(n)})$.
Moreover
\begin{eqnarray*}
&&
\Bigl|\sum_{i=1}^{k(n)} \int_0^{s} \frac{(s-\alpha)^2}{2} \M^3_{s \E\bigl[f(X)-\E[f(X)]\big|X_1,\dots,X_{i-1}\bigr]+\alpha F_i(f(X))}[F_i(f(X))] d\alpha\Bigr|
\\
&\leq& \Bigl|\sum_{i=1}^{k(n)} ||F_i(f(X))||_{\infty}^3 \int_0^{s} (s-\alpha)^2 d\alpha\Bigr|
\leq
\sum_{i=1}^{k(n)} \frac{s^3}{3} |\Delta_i f(X_1^{k(n)};Y_i)|^3
\:.
\end{eqnarray*}
On the other hand he uses the following calculation:
\begin{eqnarray*}
&&\Bigl|\frac{s^2}{2} \sum_{i=1}^{k(n)}
	\E_{s \E\bigl[f(X)-\E[f(X)]\big|X_1,\dots,X_{i-1}\bigr]}
		\bigl[\bigl(F_i(f(X))\bigr)^2\bigr]
- \frac{s^2}{2} \V f(X) \Bigr|
\\
&=& \Bigl|\frac{s^2}{2} \sum_{i=1}^{k(n)}
		\E_{s \E\bigl[f(X)-\E[f(X)]\big|X_1,\dots,X_{i-1}\bigr]}
		\bigl[\bigl(F_i(f(X))\bigr)^2\bigr]
- \frac{s^2}{2} \sum_{i=1}^{k(n)} \E\bigl[\bigl(F_i(f(X))\bigr)^2\bigr]\Bigr|
\\
&=& \frac{s^2}{2} \sum_{i=1}^{k(n)} \sum_{j=1}^{i-1}
		\E_{s \E\bigl[f(X)-\E[f(X)]\big|X_1,\dots,X_{i-1}\bigr]}
		\left[F_j\left(\bigl(F_i(f(X))\bigr)^2\right)\right]
\\
&\leq& \frac{s^2}{2}\sum_{i=1}^{k(n)} \sum_{j=1}^{i-1} \int_0^{s} \sqrt{
\E_{\alpha G_{j,i-1}}^{\mathcal G_j}[F_j\bigl(F_i^2(f(X))\bigr)^2]
\ \E_{\alpha G_{j,i-1}}^{\mathcal G_j}[W^2]} {d\alpha}
\end{eqnarray*}
applying the Cauchy-Schwartz inequality and the notation
\begin{align}
&E^{\mathcal G_j}[\cdot]:=
\E[\cdot| X_1,\dots,X_{j-1}, X_{j+1}, \dots,X_{k(n)}] \quad\text{and}\quad \nonumber
\\
&
W=G_{j,i-1}-\E_{\alpha G_{j,i-1}}^{\mathcal G_j} [G_{j,i-1}], \nonumber
\\
&\text{where }G_{j,i-1}= \E\bigl[f(X)\big| X_1,\dots,X_{i-1}\bigr]
-\E^{\mathcal G_j}\Bigl[\E\bigl[f(X)\big| X_1,\dots,X_{i-1}\bigr]\Bigr] \nonumber
\:.
\end{align}
As you can see in \cite{Catoni:2003} $F_j\bigl(F_i^2(f(X))\bigr)^2$
and $W$ can be estimated in terms of $\Delta_i f(X)$
and $\Delta_i \Delta_j f(X)$, independently of the variable of
integration $\alpha$. This leads to the inequality stated
in Theorem \ref{thmcatoni}.
\qed
\end{proofsect}
\medskip

\begin{proofsect}{Proof of Theorem~\ref{metathm}} 
To use the G\"artner-Ellis theorem (see \cite[Theorem2.3.6]{DemboZeitouni:1998})
we have to calculate the limit of
\begin{equation}
\frac{1}{s_n} \log \E
\exp \left( \lambda s_n \frac{f(X)-\E [f(X)]}{t_n} \right)
= \frac{1}{s_n} \left(
\log \E \exp \left( \frac{\lambda s_n}{t_n} f(X) \right) - \frac{\lambda s_n}{t_n} \E[f(X)]
\right)
\end{equation}
for $\lambda\in\mathbb R$.
We apply Theorem \ref{thmcatoni} for
$\displaystyle s=\frac{\lambda s_n}{t_n}$ and $\lambda>0$.
The right hand side of the inequality \eqref{catineq}
converges to zero for large $n$:
\begin{equation}
\frac{1}{s_n} s^3 d(n) = \lambda^3 \frac{s_n^2}{t_n^3} d(n)
 \stackrel{n\to\infty}{\longrightarrow} 0
\end{equation}
as assumed in condition \eqref{vier}.
Applying \eqref{catineq} this leads to the limit
\begin{equation}\label{Lambda}
\Lambda(\lambda):=
\lim_{n\to\infty}\frac{1}{s_n} \log \E
\exp \left(\lambda s_n \frac{f(X)-\E [f(X)]}{t_n} \right)
=
\lim_{n\to\infty} \frac{1}{s_n} \frac{\lambda^2 s_n^2}{2 t_n^2} \V f(X)
= \frac{\lambda^2}{2} C,
\end{equation}
where the last equality follows from condition (2).
$\Lambda$ is finite and differentiable. The same calculation is true for $-f$
and consequently \eqref{Lambda} holds for all $\lambda\in\mathbb R$. Hence
we are able to apply the G\"artner-Ellis theorem.
This proves a moderate deviation principle of
$\left(\frac{f(X)-\E [f(X)]}{t_n}\right)_n$ with speed
$s_n$ and rate function
\begin{equation*}
I(x)=
\sup_{\lambda\in\mathbb R} \left\{ \lambda x - \frac{\lambda^2}{2} C \right\}
= \frac{x^2}{2 C}\:.
\end{equation*}
\qed
\end{proofsect}

%%%%%%%%%%%%%%%%%%%%%%%%%%%%%%%%%%%%%%%%%%%%%%%%%%%%%%%%%%%%%%%%%%%%
\section{Moderate Deviations for Non-degenerate U-statistics} \label{Examples}

%%%%%%%%%%%%%%%%%%%%%%%%%%%%%%%%%%%%%%%%%%%%%%%%%%%%%%%%%%%%%%%%%%%%

In this section we show three applications of Theorem \ref{metathm}.
We start with the simplest case:

\subsection{sample mean}
Let $X_1,\dots,X_n$ be independent and identically distributed
random variables with values in a compact set $[-r,r], r>0$ fix, and
positive variance as well as $Y_1,\dots,Y_n$ independent copies. 
To apply Theorem \ref{metathm} for $f(X)=\frac{1}{\sqrt{n}} \sum_{m=1}^n X_m$
the partial differences of $f$ have to tend to zero fast enough for n to
infinity:
\begin{align}
&|\Delta_i f(X_1^n; Y_i)|
= \frac{1}{\sqrt{n}} |X_i-Y_i| \leq \frac{2r}{\sqrt{n}}
\\
&\Delta_i \Delta_j f(X_1^n; Y_j,Y_i)=0
\end{align}
Let $a_n$ be a sequence with $\lim_{n\to\infty}\frac{\sqrt{n}}{a_n} = 0$
and $\lim_{n\to\infty} \frac{n}{a_n} =\infty$. For $t_n= \frac{a_n}{\sqrt{n}}$
and $s_n=\frac{a_n^2}{n}$ the conditions of Theorem \ref{metathm} are
satisfied:
\begin{enumerate}
\item
$\displaystyle \frac{s_n^2}{t_n^3} d(n)
\leq \frac{a_n}{\sqrt{n}} \frac{4 r^2}{n}  \sum_{m=1}^n \frac{2r}{3\sqrt{n}}
=   \frac{a_n}{n} \frac{8 r^3}{3}.$
Because $d(n)$ is positive this implies
$\displaystyle
\lim_{n \to \infty} \frac{s_n^2}{t_n^3} d(n)=0$.
\item
$\displaystyle \frac{s_n}{t_n^2} \V f(X) = \V\left(\frac{1}{\sqrt{n}}\sum_{m=1}^n X_m\right) = \V(X_1)$.
\end{enumerate}
The application of Theorem \ref{metathm} proves the MDP for
$\displaystyle
\frac{1}{a_n} \left(\sum_{m=1}^n X_m - n\E X_1\right)_n$
with speed $s_n$ and rate function $I(x)=\frac{x^2}{2 \V X_1}$.
This result is well known, see for example \cite{DemboZeitouni:1998}, Theorem 3.7.1,
and references therein. The MDP can be proved under {\it local exponential moment}
conditions on $X_1$: $\E (\exp (\lambda X_1)) < \infty$ for a $\lambda >0$. In \cite{Catoni:2003},
the bounds of the $\log$-Laplace transformation are obtained under exponential moment
conditions. Applying this result, we would be able to obtain the MDP under exponential moment
conditions, but this is not the focus of this paper. 

\subsection{non-degenerate U-statistics with independent entries}
Let $X_1,\dots,X_n$ be independent and identical distributed random
variables with values in a measurable space $\mathcal X$. For a
measurable and symmetric function $h:{\mathcal X}^m\to \R$ we define
$$
U_n(h):= \frac{1}{\left(n\atop m\right)}
\sum_{1\leq i_1<\dots<i_m \leq n} h(X_{i_1},\dots,X_{i_m})\:,
$$
where symmetric means invariant under all permutation of its arguments.
$U_n(h)$ is called a {\it U-statistic} with {\it kernel} $h$ and
{\it degree} $m$.

Define the conditional expectation for $c=1,\dots,m$ by
\begin{eqnarray*}
h_c(x_1,\dots,x_c)
&:=&\E\bigl[ h(x_1,\dots,x_c,X_{c+1},\dots, X_m)\bigr]
\\
&=&\E\bigl[ h(X_1,\dots, X_m)\big| X_1=x_1,\dots, X_c=x_c\bigr]
\end{eqnarray*}
and the variances by $\sigma_c^2:=\V\bigl[h_c(X_1,\dots,X_c)\bigr]$.
A U-statistic is called {\it degenerate of order $d$} if and only if $0=\sigma_1^2 = \cdots = \sigma_d^2 < \sigma_{d+1}^2$ and
and {\it non-degenerate} if $\sigma_1^2>0$.

By the Hoeffding-decomposition (see for example \cite{Lee:1990}), we know
that for every symmetric function $h$, the $U$-statistic can be decomposed into a sum
of degenerate $U$-statistics of different orders. 
In the degenerate case the linear term of this decomposition
disappears. Eichelsbacher and Schmock showed the MDP for non-degenerate $U$-statistics
in \cite{EichelsbacherSchmock:2003}; the proof used the fact that the linear term 
in the Hoeffding-decomposition is leading in the non-degenerate case. 
In this article the observed U-statistic is
assumed to be of the latter case.

We show the MDP for appropriate scaled U-statistics without applying Hoeffding's decomposition.
The scaled U-statistic $f:=\sqrt{n} U_n(h)$ with bounded kernel $h$ and
degree $2$ fulfils the inequality:
\begin{eqnarray*}
\lefteqn{\Delta_k f(x_1^n; y_k)
= \frac{2 \sqrt{n}}{n(n-1)} \Big( \sum_{1\leq i<j\leq n} h(x_i,x_j)
- \sum_{{1\leq i<j\leq n}\atop{i,j\not=k}} h(x_i,x_j)
- \sum_{i=1}^{k-1} h(x_i,y_k)}
\\
&&{}-\sum_{j=k+1}^n h(y_k,x_j)
\Big)
\\
&=&\frac{2}{\sqrt{n}(n-1)}
\left(
\sum_{i=1}^{k-1} h(x_i,x_k) +\sum_{j=k+1}^n h(x_k,x_j)
-\sum_{i=1}^{k-1} h(x_i,y_k) -\sum_{j=k+1}^n h(y_k,x_j)
\right)\\
&\leq&
\frac{4 ||h||_{\infty}}{\sqrt{n}}
\end{eqnarray*}
for $k=1,\dots,n$.  Analogously one
can write down all summations of the kernel $h$ for $\Delta_m \Delta_k
f(x_1^n; y_k,y_m)$. Most terms add up to zero and we get:
\begin{eqnarray*}
\Delta_m \Delta_k f(x_1^n; y_k,y_m)
&=& \frac{2 \left(h(x_k,x_m)-h(y_k,x_m)-h(x_k,y_m)+h(y_k,y_m) \right)}{\sqrt{n}(n-1)}
\\
&\leq& \frac{2}{\sqrt{n}(n-1)} 4 ||h||_{\infty}
\leq \frac{16 ||h||_{\infty}}{n^{3/2}}\:.
\end{eqnarray*}

Let $a_n$ be a sequence with $\lim_{n\to\infty}\frac{\sqrt{n}}{a_n} = 0$
and $\lim_{n\to\infty} \frac{n}{a_n} =\infty$.
The aim is the MDP for a real random variable of the kind $\frac{n}{a_n}
U_n(h)$ and the speed $s_n:= \frac{a_n^2}{n}$.  
To apply Theorem \ref{metathm} for $f(X)=\sqrt{n} U_n(h)(X)$, $s_n$
as above and $t_n:= \frac{a_n}{\sqrt{n}}$, we obtain
\begin{enumerate}
\item
$\displaystyle \frac{s_n^2}{t_n^3} d(n)
\leq \frac{a_n}{\sqrt{n}} \left(
\frac{4 ||h||_{\infty}^3}{3 \sqrt{n}} + \frac{n-1}{n^{3/2}} 8 ||h||_{\infty}^3
\right)$. The right hand side converges to 0, because $\lim_{n \to \infty} a_n/n =0$.
\item
$\displaystyle \frac{s_n}{t_n^2} \V f(X)
=\frac{a_n^2}{n} \frac{n}{a_n^2} \V\bigl(\sqrt{n} U_n(h)(X)\bigr)
\stackrel{n\to\infty}{\longrightarrow} \ 4\sigma_1^2
$,  see Theorem 3 in \cite[chapter 1.3]{Lee:1990}.
\end{enumerate}
The non-degeneracy of $U_n(h)$ implies that $4\sigma_1^2>0$.

The application of Theorem \ref{metathm} proves:

\begin{theorem}\label{thmustat}
Let $(a_n)_n\in (0,\infty)^{\mathbb N}$ be a sequence with
$\lim_{n\to\infty}\frac{\sqrt{n}}{a_n} = 0$ and
$\lim_{n\to\infty} \frac{n}{a_n} =\infty$.
Then the sequence of non-degenerate and centered U-statistics $\bigl( \frac{n}{a_n} U_n(h) \bigr)_n$ 
with a real-valued, symmetric and bounded kernel function $h$ satisfies the
MDP with speed $s_n:= \frac{a_n^2}{n}$ and good rate function
$$
I(x)
= \sup_{\lambda\in{\mathbb R}}\{\lambda x - 2 \lambda^2 \sigma_1^2\}
= \frac{x^2}{8 \sigma_1^2}\:.
$$
\end{theorem}

\begin{remark}
Theorem \ref{thmustat} holds, if the kernel function $h$ depends on $i$ and
$j$, e.g. the $U$-statistic is of the form
$\frac{1}{\left(n \atop 2\right)}
\sum_{1\leq i<j\leq n} h_{i,j}(X_i,X_j)$. One can see this in the estimation of
$\Delta_i f(X)$ and $\Delta_i \Delta_j f(X)$. This is an improvement of the 
result in \cite{EichelsbacherSchmock:2003}.
\end{remark}
\begin{remark}
We considered U-statistics with degree $2$. For degree $m>2$ we get the
following estimation for the partial differences of
\begin{align}
&f(X):= \frac{1}{\sqrt{n} \left({n}\atop{m}\right)}
\sum_{1\leq i_1<\dots< i_m\leq n} h(X_{i_1},\dots,X_{i_m})\:: \nonumber
\\
&\Delta_i f(X)
\leq \sqrt{n} \frac{1}{\left({n}\atop{m}\right)}
\left({n-1}\atop{m-1}\right) 2 ||h||_{\infty}
=\frac{2 m}{\sqrt{n}}  ||h||_{\infty} \nonumber
\\
&\Delta_i \Delta_j f(X)
\leq \sqrt{n} \frac{1}{\left({n}\atop{m}\right)}
\left({n-2}\atop{m-2}\right) 4 ||h||_{\infty}
= \frac{4 m (m-1)}{\sqrt{n} (n-1)}  ||h||_{\infty} \nonumber
\end{align}
and Theorem \ref{metathm} can be applied as before.
\end{remark}

Theorem \ref{thmustat} is proved in \cite{EichelsbacherSchmock:2003} in a more
general context. Eichelsbacher and Schmock showed a moderate deviation
principle for degenerate and non-degenerate U-statistics with a kernel
function $h$, which is bounded or satisfies exponential moment conditions
(see also \cite{Eichelsbacher:1998, Eichelsbacher:2001}).

\medskip

{\it Example 1:} Consider the {\it sample variance} $U_n^{\V}$,
which is a U-statistic of degree $2$ with kernel
$h(x_1,x_2)=\frac{1}{2}(x_1-x_2)^2$. Let the random variables
$X_i, i=1,\dots,n,$ be restricted to take values in a compact interval.
A simple calculation shows
$$
\sigma_1^2
= \V\bigl[ h_1(X_1)\bigr]
=\frac{1}{4}\V\bigl[(X_1-\E X_1)^2 %+\V(X_1)
\bigr]
=\frac{1}{4}
\left( \E[(X_1-\E X_1)^4] - (\V X_1)^2 \right)\:.
$$
The U-statistic is non-degenerate, if the condition $\E[(X_1-\E
X_1)^4] > (\V X_1)^2$ is satisfied.
Then $\left(\frac{n}{a_n(n-1)} \sum_{i=1}^n (X_i-\bar{X})^2\right)_n$
satisfies the MDP with speed $\frac{a_n^2}{n}$ and good rate function
$$
I^{\V}(x)= \frac{x^2}{8\sigma_1^2}
=\frac{x^2}{2 \left(\E[(X_1-\E X_1)^4]-(\V X_1)^2\right)}\:.
$$

In the case of independent Bernoulli random variables with
$P(X_1=1)=1-P(X_1=0)=p$, $0<p<1$, $U_n^{\V}$ is a non-degenerate
U-statistic for $p\not= \frac{1}{2}$ and the corresponding rate
function is given by:
$$
I^{\V}_ {\text{bernoulli}}(x)
=\frac{x^2}{2 p(1-p)\bigl(1-4p(1-p)\bigr)}\:.
$$

{\it Example 2:} The {\it sample second moment} is defined by the
kernel function $h(x_1,x_2)=x_1 x_2$. This leads to
$$\sigma_1^2
=\V\bigl( h_1(X_1)\bigr) = \V\bigl(X_1 \E X_1 \bigr)
=\bigl( \E X_1\bigr)^2 \V X_1\:.
$$
The condition $\sigma_1^2>0$ is satisfied, if the expectation and the
variance of the observed random variables are unequal to zero. The values
of the random variables have to be in a compact interval as in the
example above.
Under this conditions $\frac{n}{a_n} \sum_{1\leq i< j \leq n} X_i X_j$
satisfies the MDP with speed $\frac{a_n^2}{n}$ and good rate function
$$
I^{\text{sec}}(x)
=\frac{x^2}{8\sigma_1^2} =\frac{x^2}{8 \bigl( \E X_1\bigr)^2 \V X_1}\:.
$$

For independent Bernoulli random variables the rate function for all
$0<p<1$ is:
$$
I^{\text{sec}}_{\text{bernoulli}}(x)
=\frac{x^2}{8 p^3 (1-p)}\:.
$$

{\it Example 3:} {\it Wilcoxon one sample statistic}
Let $X_1,\dots,X_n$ be real valued, independent and identically
distributed random variables with absolute continuous distribution
function symmetric in zero.
We prove the MDP for -properly rescaled- 
\begin{equation*}
W_n= \sum_{1\leq i<j \leq n} 1_{\{X_i+X_j>0\}} 
=\left(n\atop2\right) U_n(h)
\end{equation*}
defining $h(x_1,x_2):=  1_{\{x_1+x_2>0\}}$ for all $x_1,x_2\in\mathbb R$.
Under these assumptions one can calculate
$\sigma_1^2=\text{Cov}\bigl(h(X_1,X_2),h(X_2,X_3)\bigr) = \frac{1}{12}$.
Applying Theorem \ref{thmustat} as before we proved the MDP
for the Wilcoxon one sample statistic
$\frac{1}{(n-1) a_n} \left(W_n-\frac{1}{2}\left(n\atop2\right)\right)$
with speed $\frac{a_n^2}{n}$ and good rate function
$I^W(x)= \frac{3}{2} x^2$.

\subsection{non-degenerate U-statistics with Markovian entries}\label{Markov}
The moderate deviation principle in Theorem \ref{metathm} is stated for
independent random variables. Catoni showed in \cite{Catoni:2003}, that
the estimation of the logarithm of the Laplace transform can be
generalized for Markov chains via a coupled process.
In the following one can see, that these results yield analogously to the
proof of Theorem \ref{metathm} to a moderate deviation principle.

In this section we use the notation introduced in \cite{Catoni:2003},
Chapter 3.

Let us assume that $(X_k)_{k\in{\mathbb N}}$ is a Markov chain
such that for $X:=(X_1,\dots, X_n)$ the following inequalities hold
\begin{align}
P\bigl(\tau_i > i+k\big| {\mathcal G}_i, X_i\bigr)
&\leq A \rho^k \quad \forall k\in{\mathbb N} \quad a.s.  \label{Cat01} 
\\
P\bigl(\tau_i > i+k\big| {\mathcal F}_n, {\stackrel{i}{Y}}_i\bigr)
&\leq A \rho^k \quad \forall k\in{\mathbb N} \quad a.s.  \label{Cat02}
\end{align}
for some positive constants $A$ and $\rho<1$. 
Here ${\stackrel{i}{Y}} := ({\stackrel{i}{Y}}_1, \ldots, {\stackrel{i}{Y}}_n)$, $i=1, \ldots, n$, 
are $n$ coupled stochastic processes satisfying for any $i$ that 
${\stackrel{i}{Y}}$ is equal in distribution to $X$. For the list of the properties of these coupled processes, 
see page 14 in \cite{Catoni:2003}. Moreover, the $\sigma$-algebra $ {\mathcal G}_i$ in \eqref{Cat01} is generated
by ${\stackrel{i}{Y}}$, the $\sigma$-algebra ${\mathcal F}_n$ in \eqref{Cat02} is generated by $(X_1, \ldots, X_n)$.
Finally the coupling stopping times $\tau_i$ are defined as
$$
\tau_i = \inf \{ k \geq i | {\stackrel{i}{Y}}_k = X_k \, \}.
$$
Now we can state our result:

\begin{theorem}\label{thmmarkov}
Let us assume that $(X_k)_{k\in{\mathbb N}}$ is a Markov chain
such that for $X:=(X_1,\dots, X_n)$ \eqref{Cat01} and \eqref{Cat02} hold true.
Let $U_n(h)(X)$ be a
non-degenerate U-statistic with bounded kernel function $h$ and
$\lim_{n\to\infty} \V\bigl(\sqrt{n}U_n(h)(X)\bigr)<\infty$.
Then for every sequence $a_n$, where
$$\lim_{n\to\infty} \frac{a_n}{n}=0 \text{ and }
\lim_{n\to\infty} \frac{n}{a_n^2}=0\:,$$
the sequence $\bigl( \frac{n}{a_n} U_n(h)(X) \bigr)_n$
satisfies a moderate deviation principle with speed
$s_n=\frac{a_n^2}{n}$ and rate function I given by
$$
I(x):= \sup_{\lambda\in\mathbb R}\left\{ \lambda x - \frac{\lambda^2}{2} \lim_{n\to\infty} \V\bigl(\sqrt{n} U_n(h)(X)\bigr)\right\}\:.
$$
\end{theorem}

\begin{proof} 
As for the independent case we define
$f(X):=\sqrt{n}U_n(h)(X_1,\dots,X_n)$.
Corollary 3.1 of \cite{Catoni:2003} states, that in the
above situation the inequality
\begin{align}
&\Bigl|\log{\E\exp{\bigl(s f(X)\bigr)}}-s\E\bigl[f(X)\bigr]
-\frac{s^2}{2}\V f(X)\Bigr| \nonumber
\\
\leq&
\frac{s^3}{\sqrt{n}} \frac{BCA^3}{(1-\rho)^3} \left(\frac{\rho \log{(\rho^{-1})}}{2 AB} -\frac{s}{\sqrt{n}}\right)_+^{-1} \nonumber
\\
&{}+
\frac{s^3}{\sqrt{n}} \left(\frac{B^3 A^3}{3(1-\rho)^3}
+\frac{4 B^2 A^3}{(1-\rho)^3}\left(\frac{\rho \log{(\rho^{-1})}}{2 AB} -\frac{s}{\sqrt{n}}\right)_+^{-1}
\right) \nonumber
\end{align}
holds for some constants B and C.
This is the situation of Theorem \ref{metathm} except that in this case
$d(n)$ is defined by
\begin{equation*}
\frac{1}{\sqrt{n}} \frac{BCA^3}{(1-\rho)^3} \left(\frac{\rho \log{(\rho^{-1})}}{2 AB} -\frac{s}{\sqrt{n}}\right)_+^{-1}
+
\frac{1}{\sqrt{n}} \left(\frac{B^3 A^3}{3(1-\rho)^3}
+\frac{4 B^2 A^3}{(1-\rho)^3}\left(\frac{\rho \log{(\rho^{-1})}}{2 AB} -\frac{s}{\sqrt{n}}\right)_+^{-1}
\right)\:.
\end{equation*}
This expression depends on $s$. We apply the adapted Theorem \ref{metathm} for
$s_n=\frac{a_n^2}{n}$, $t_n:= \frac{a_n}{\sqrt{n}}$ and
$s:=\lambda \frac{a_n}{\sqrt{n}}$ as before. 

Because of
$\frac{s}{\sqrt{n}}= \lambda \frac{a_n}{n}
\stackrel{n\to\infty}{\longrightarrow} 0$,
the assumptions of Theorem \ref{metathm} are satisfied:
\begin{enumerate}
\item
$\frac{s_n^2}{t_n^3} d(n)
= \frac{a_n}{n} \frac{BCA^3}{(1-\rho)^3} \left(\frac{\rho \log{(\rho^{-1})}}{2 AB} -\frac{s}{\sqrt{n}}\right)_+^{-1}
+
\frac{a_n}{n} \left(\frac{B^3 A^3}{3(1-\rho)^3}
+\frac{4 B^2 A^3}{(1-\rho)^3}\left(\frac{\rho \log{(\rho^{-1})}}{2 AB} -\frac{s}{\sqrt{n}}\right)_+^{-1}\right)$\\
\quad$\stackrel{n\to\infty}{\longrightarrow} 0 \:.$
\item
$\frac{s_n}{t_n^2} \V f(X) = \V\bigl(\sqrt{n} U_n(h)(X)\bigr)<\infty$ as assumed.
\end{enumerate}
Therefore we can use the G\"artner-Ellis theorem to prove a moderate
deviation principle for $(\frac{n}{a_n} U_n(h)(X))_n$.
\end{proof}

\begin{cor}
Let $(X_k)_{k\in{\mathbb N}}$ be a strictly stationary, aperiodic and
irreducible Markov chain which finite state space and
$U_n(h)(X)$ be a non-degenerate U-statistic based on
a bounded kernel $h$ of degree two. 
Then $(\frac{n}{a_n} U_n(h)(X))_n$ satisfies the MDP
with speed and rate function as in Theorem \ref{thmmarkov}.
\end{cor}

\begin{proof}
The Markov chain is strong mixing and the absolute
regularity coefficient $\beta(n)$ converges to 0 at least exponentially fast as
$n$ tends to infinity, see \cite{Bradley:2005}, Theorem 3.7(c). Hence the
equations \eqref{Cat01} and \eqref{Cat02} are satisfied and Theorem
\ref{thmmarkov} can be applied. The limit of the variance of $\sqrt{n}
U_n(h)$ is bounded, see \cite{Lee:1990}, 2.4.2 Theorem 1, which
proves the MDP for this example.
\end{proof}
\medskip

For Doeblin recurrent and aperiodic Markov chains the MDP for additive functionals of
a Markov process is proved in
\cite{Wu:1995}. In fact Wu proves the MDP under the condition that
$1$ is an isolated and simple eigenvalue of
the transition probability kernel satisfying that it is the only
eigenvalue with modulus $1$.
For a continuous spectrum of the transition probability kernel
Delyon, Juditsky and Lipster present in \cite{DelyonJuditskyLipster:2006}
a method for objects of the form
$$
\frac{1}{n^{\alpha}} \sum_{i=1}^n H(X_{i-1}), \,\, \frac{1}{2}<\alpha<1, n\geq 1,
$$
where $(X_i)_{i\geq 0}$ is a homogeneous ergodic Markov chain
and the vector-valued function $H$ satisfies a Lipschitz continuity.
To the best of our knowledge, we proved the first MDP for a $U$-statistic
with Markovian entries.

%%%%%%%%%%%%%%%%%%%%%%%%%%%%%%%%%%%%%%%%%%%%%%%%%%%%%%%%%%%%%
\section{Proof of Theorem \ref{theorem1} and \ref{concentration}}\label{sectionsubgraphs}

\begin{lemma}\label{lemma1}
The standardized subgraph count statistic $Z$ satisfies
the inequalities
\begin{eqnarray}\label{dsub}
&\Delta_i Z \leq \frac{1}{\sqrt{\left({{n}\atop{2}}\right)p(1-p)} p^{k-1}}
\\\label{ddsub}
&\sum_{i=1}^{\left(n\atop2\right)} \sum_{j=1}^{i-1} \Delta_i \Delta_j Z \leq
\frac{1}{c_{n,p}} \left(n\atop2\right)\left({n-2}\atop{l-2}\right)
(l-2)^2 (l-2)!
\end{eqnarray}
\end{lemma}

\begin{proofsect}{Proof of Lemma \ref{lemma1}}
As the first step we will find an upper bound for
\begin{equation*}
\Delta_i Z = Z- \frac{1}{c_{n,p}} \sum_{1\leq \kappa_1<\dots < \kappa_k \leq \left({{n}\atop{2}}\right)}
  1_{\{(e_{\kappa_1},\dots,e_{\kappa_k})\sim G\}} \left(
    \prod_{j=1}^k X_{i, \kappa_j} - p^k\right)\:,
\end{equation*}
where $(X_{i,1}, X_{i,2}, \dots,X_{i, \left({{n}\atop{2}}\right)})=(X_1,\dots,X_{i-1},Y_i,X_{i+1},\dots,X_{\left({{n}\atop{2}}\right)})$
and $Y_i$ is an independent copy of $X_i$,
$i\in\{1, \dots,\left({{n}\atop{2}}\right)\}$.
The difference consists only of those summands which contain the random variable $X_i$ or $Y_i$.
The number of subgraphs isomorphic to $G$ and containing a fixed edge,
is given by
$$
\left({{n-2}\atop{l-2}}\right)  \frac{2k}{a} (l-2)! \:,
$$
see \cite{NowickiWierman:1988}, p.307. Therefore we can estimate
\begin{equation}\label{Deltazaehlgraph}
\big|\Delta_i Z\big| \leq
\frac{1}{\sqrt{\left({{n}\atop{2}}\right)p(1-p)} p^{k-1}}\:.
\end{equation}

For the second step we have to bound the partial difference of order two of the
subgraph count statistic.
\begin{eqnarray*}
&&\Delta_i \Delta_j Z
\\
&=&
\frac{1}{c_{n,p}}
\sum_{
{1\leq \kappa_1<\dots<\kappa_{k-2}}\atop{\kappa_1,\dots,\kappa_{k-2}\not= i,j}
}
1_{\{(e_i,e_j,e_{\kappa_1},\dots,e_{\kappa_{k-2}})\sim G\}}
\prod_{m=1}^{k-2} X_{\kappa_m} X_j (X_i-Y_i)
\nonumber\\
&&{}-
\frac{1}{c_{n,p}}
\sum_{
{1\leq \kappa_1<\dots<\kappa_{k-2}}\atop{\kappa_1,\dots,\kappa_{k-2}\not= i,j}
}
1_{\{(e_i,e_j,e_{\kappa_1},\dots,e_{\kappa_{k-2}})\sim G\}}
\prod_{m=1}^{k-2} X_{\kappa_m} Y_j (X_i-Y_i)
\\
&=&
\frac{1}{c_{n,p}}
\sum_{
{1\leq \kappa_1<\dots<\kappa_{k-2}}\atop{\kappa_1,\dots,\kappa_{k-2}\not= i,j}
}
1_{\{(e_i,e_j,e_{\kappa_1},\dots,e_{\kappa_{k-2}})\sim G\}}
\prod_{m=1}^{k-2} X_{\kappa_m} (X_j-Y_j) (X_i-Y_i)
\end{eqnarray*}
Instead of directly bounding the random variables we first care on
cancellations due to the indicator function.
We use the information about the fixed graph $G$.
To do this we should distinguish the case, whether $e_i$
and $e_j$ have a common vertex.
\begin{itemize}
\item $e_i$ and $e_j$ have a common vertex:\\
Because $G$ contains $l$ vertices, we have
$\left( {n-3}\atop{l-3}\right)$ possibilities to fix all vertices
of the subgraph isomorph to $G$ and including the edges $e_i$ and $e_j$.
The order of the vertices is important and so we have to take the factor
$2(l-2)!$ into account.

\item $e_i$ and $e_j$ have four different vertices:\\
Four fixed vertices allow us to choose only 
$\left({n-4}\atop{l-4}\right)$ more. The order of the vertices and the
relative position of $e_i$ and $e_j$ are relevant. So as before the
factor is given by $2(l-2)!$.
\end{itemize}

Bounding the random variables $X_i,Y_i$,
$i\in\{1,\dots,\left(n\atop2\right)\}$ by $1$,
we achieve the following estimation:
\begin{eqnarray}
&&\sum_{i=1}^{\left(n\atop2\right)} \sum_{j=1}^{i-1} \Delta_i \Delta_j Z
\label{doppeldelta}\\
&\leq&
\frac{\left(n\atop2\right)}{c_{n,p}}  \left(
4(n-2)\left( {n-3}\atop{l-3}\right)
+ (n-2)(n-3)\left({n-4}\atop{l-4}\right)
\right) (l-2)!
\label{ineqdoppeldelta}\\
&=&\frac{1}{c_{n,p}} \left(n\atop2\right)\left({n-2}\atop{l-2}\right)
(l-2)^2 (l-2)!\ .\label{schritt2zaehlst}
\end{eqnarray}
To bound $\sum_{j=1}^{i-1} \Delta_i \Delta_j Z$ for $i$ fixed
in \eqref{doppeldelta} one has to observe that there
are at most $2(n-2)$ indices $j<i$, such that $e_i$ and $e_j$
have a common vertex, and
$\frac{1}{2} (n-2)(n-3)=\left(n\atop2\right) - 2(n-2) -1$
indices $j$, such that $e_i$ and $e_j$ have no common vertex.
This proves the inequality \eqref{ineqdoppeldelta} and hence
\eqref{schritt2zaehlst} follows.
\qed
\end{proofsect}

To apply Theorem \ref{metathm} we choose
$s_n= \frac{\left(\frac{2k}{a} (l-2)!\right)^2 \beta_n^2}{c_{n,p}^2}$
and $t_n= \frac{\beta_n}{c_{n,p}}$. 
\begin{equation*}
\frac{s_n}{t_n^2} \V Z
= \left(\frac{2k}{a} (l-2)!\right)^2 \V Z
\stackrel{n\to\infty}{\longrightarrow} \left(\frac{2k}{a} (l-2)!\right)^2 
\:,
\end{equation*}
because $\lim_{n\to\infty} \V Z=1$, see \cite{NowickiWierman:1988}.
We need Lemma \ref{lemma1} to bound $d(n)$:
\begin{eqnarray} \label{dnbound}
d(n)&\stackrel{\eqref{dsub}}{\leq}&
\frac{1}{ \left({n}\atop{2}\right) p^{2k-1} (1-p)}
\sum_{i=1}^{\left({n}\atop{2}\right)}
\left(
\frac{1}{3 \sqrt{\left({n}\atop{2}\right)} p^{k-1/2} (1-p)^{1/2}}
+
\sum_{j=1}^{i-1} \Delta_i \Delta_j Z
\right) \nonumber
\\
&\stackrel{\eqref{ddsub}}{\leq}&
\frac{1}{ \left({n}\atop{2}\right) p^{2k-1} (1-p)}
\left(
\frac{\sqrt{\left({n}\atop{2}\right)}}{3 p^{k-1/2} (1-p)^{1/2}}
+
\frac{ \left(n\atop2\right)\left({n-2}\atop{l-2}\right)
(l-2)^2 (l-2)!}{c_{n,p}}
\right) \nonumber
\\
&=&\frac{1}{\sqrt{\left({n}\atop{2}\right)}}
\frac{1}{p^{3(k-1/2)} (1-p)^{3/2}}
\left(\frac{1}{3}+ \frac{(l-2)^2 a}{2k}\right)
\end{eqnarray}
And condition \ref{drei} of Theorem \ref{metathm} follows from
\begin{align}
\frac{s_n^2}{t_n^3} d(n)
&\leq
\beta_n \frac{1}{\left({n-2}\atop{l-2}\right) \left({n}\atop{2}\right)}
\frac{1}{p^{4k-2} (1-p)^2} \left(\frac{1}{3}+ \frac{(l-2)^2 a}{2k}\right)
\left(\frac{2k}{a}(l-2)!\right)^3 \nonumber
\\
&
\stackrel{n\to\infty}{\longrightarrow} 0\quad\text{, if}\quad
\beta_n \ll n^{l} p^{4k-2} (1-p)^2\text{ as assumed.} \nonumber
\end{align}
$\frac{s_n^2}{t_n^3} d(n)$ is positive and therefore the limit
of $n$ to infinity is zero, too. With Theorem \ref{metathm} we proved Theorem
\ref{theorem1}.

\begin{remark}\label{CLTsubgraph}
The estimation by Catoni, see Theorem \ref{thmcatoni}, and Lemma \ref{lemma1}
allow us to give an alternative proof for the central limit theorem of the
subgraph count statistic $Z$, if
$n p^{3(k-\frac{1}{2})}\stackrel{n\to\infty}{\longrightarrow}  \infty$
and $n^2 (1-p)^{3/2} \stackrel{n\to\infty}{\longrightarrow} \infty$.
On these conditions it follows, that $d(n)\stackrel{n\to\infty}{\longrightarrow}0$,
and it is easy to calculate the following limits:
\begin{eqnarray*}
&&\lim_{n\to\infty} \E e^{\lambda Z} = e^{\frac{\lambda^2}{2} \lim_{n\to\infty} \V Z}
= e^{\frac{\lambda^2}{2}} <\infty \quad \text{for all} \, \lambda>0
\\
&&\text{and additionally}\quad
\lim_{\lambda \nearrow 0} \lim_{n\to\infty} \E e^{\lambda Z}= 1\:.
\end{eqnarray*}
Hence the central limit theorem results from the continuity theorem.
Both conditions are stronger than the one in \cite{NowickiWierman:1988}.
\end{remark}

\begin{proofsect}{Proof of Theorem \ref{concentration}}
We apply Theorem \ref{thmcatoni} and the Chebychev inequality to get
$$
P(Z \geq \varepsilon ) \leq 
\exp\left(-s \varepsilon + \frac{s^2}{2} \V Z + s^3 d(n)\right)
$$
for all $s>0$ and all $\varepsilon >0$.
Choosing $s=\frac{\varepsilon}{\V Z +\frac{2 d(n) \varepsilon}{\V Z}}$
implies
\begin{equation} \label{conz1}
P(Z \geq \varepsilon)
\leq
\exp\left(-\frac{\varepsilon^2}{2 \bigl( \V Z + \frac{ 2 d(n) \varepsilon}{\V Z}\bigr)}\right)\:.
\end{equation}
Applying Theorem \ref{thmcatoni} to $-Z$ gives
$$
P(Z \leq - \varepsilon)
\leq
\exp\left(-\frac{\varepsilon^2}{2 \bigl( \V Z + \frac{ 2 d(n) \varepsilon}{\V Z}\bigr)}\right)\:.
$$
Now we consider an upper bound for the upper tail $P(W -\E W \geq \varepsilon \, \E W) = P\bigl( Z \geq \frac{\varepsilon \, \E W}{c_{n,p}} \bigr)$.
Using $\V Z = c_{n,p}^{-2} \V W$, inequality \eqref{conz1} leads to
$$
P( W - \E W \geq \varepsilon \, \E W) \leq \exp \biggl( - \frac{\varepsilon^2 (\E W)^2}{2 \V W + 4 \varepsilon d(n) \E W c_{n,p}^3 (\V W)^{-1}} \biggr).
$$ 
Indeed, this concentration inequality holds for $f(X) - \E f(X)$ in Theorem \ref{metathm} with $d(n)$ given as 
in \eqref{defdn}.
We restrict our calculations to the subgraph-counting statistic $W$. We will use the following bounds for
$\E W$, $\V W$ and $c_{n,p}$: there are constants, depending only on $l$ and $k$, such that
$$
const. \, n^l \, p^k \leq \E W \leq const. \, n^l \, p^k,
$$
$$
const. \, n^{2l-2} p^{2k-1} (1-p) \leq \V W \leq const. \, n^{2l-2} p^{2k-1} (1-p)
$$
(see \cite[2nd section]{Rucinski:1988}), and
$$
const.\, n^{l-1} p^{k-1/2} (1-p)^{1/2} \leq c_{n,p} \leq const. \, n^{l-1} p^{k-1/2} (1-p)^{1/2}.
$$
Using the upper bound \eqref{dnbound} for $d(n)$, we obtain
$$
P( W - \E W \geq \varepsilon \, \E W) \leq \exp \biggl( -\frac{const. \varepsilon^2 n^{2l} p^{2k}}{n^{2l-2} p^{2k-1} (1-p) + 
const. \varepsilon n^{2l-2} p^{-k+1} (1-p)^{-1}} \biggr),
$$
which proves Theorem \ref{concentration}.
\qed
\end{proofsect}

\bigskip
{\it Acknowledgment:}
The first author has been supported by {\it Studienstiftung des deutschen Volkes.}
\nocite{*}
\bibliographystyle{alpha}
\bibliography{lit2}

\end{document}